\documentclass[11pt]{amsart}
\usepackage[T1]{fontenc}
\usepackage{lmodern}
\usepackage[margin=1in]{geometry}
\usepackage{amssymb,mathtools,microtype,xcolor}
\usepackage[colorlinks=true,linkcolor=blue!45!black,citecolor=blue!45!black,urlcolor=blue!45!black]{hyperref}
\hypersetup{pdftitle={Distinguished standard modules in the Archimedean case (Gal1)},pdfauthor={Alan Xuelun Hou and Tudor Popescu}}
\newcommand{\C}{\mathbb C}
\newcommand{\R}{\mathbb R}
\newcommand{\HH}{\mathbb H}
\DeclareMathOperator{\GL}{GL}
\DeclareMathOperator{\Hom}{Hom}
\DeclareMathOperator{\Ad}{Ad}
\DeclareMathOperator{\diag}{diag}
\DeclareMathOperator{\Lie}{Lie}
\newcommand{\IP}{I_P^G}
\newtheorem{theorem}{Theorem}
\newtheorem{lemma}[theorem]{Lemma}
\numberwithin{equation}{section}

\title[Distinguished standard modules in case (Gal1)]{Distinguished standard modules for $\GL_{2m}(\mathbb{C})/\GL_{m}(\mathbb{H})$}
\author{Alan Xuelun Hou}
\author{Tudor Popescu}
\date{}
\begin{document}
\begin{abstract}
We characterize the standard modules of $\GL_{2m}(\C)$ that
are distinguished by $\GL_m(\HH)$. Let $\delta_1,\dots, \delta_{2m}$ be characters of $\mathbb{C}^\times$. Assume that $S = \delta_1 \times \cdots \times \delta_{2m}$ is a standard module of $\GL_{2m}(\mathbb{C})$. For each $i$, define ${\delta_i^*} (z) = \delta_i(\overline{z})^{-1}$ for $ z \in \mathbb{C}^\times$. In particular, we conclude that a standard module for $\GL_{2m}(\mathbb{C})$ is distinguished by $\GL_{m}(\mathbb{H})$ if and only if there
exists an involution $p\in S_{2m}$ without fixed points such that
$\delta_{p(i)}=\delta_i^*$ for every $i$.
We first verify the hypotheses of the multiplicity estimate theorem of Suzuki
and Tamori in \cite{ST}.
The orbit calculation of Matringe, Offen, and Yang in \cite{MOYglobal} then gives the necessary condition and a dimension bound.
Then local intertwining periods prove sufficiency.
\end{abstract}
\maketitle

\section{Introduction}

We use the notation of Sections 2 and 4 in \cite{MOYglobal} for
the Archimedean case (denoted as Gal1 in their work).
Let $m\ge1$ be an integer.
Set $G=\GL_{2m}(\C)$ and $H=\GL_m(\HH)$.
Write $\HH=\C\oplus\C\mathbf j$, where $\mathbf j^2=-1$ and
$\mathbf jz=\bar z\mathbf j$ for $z\in\C$.
Define
\[
 \phi(a+b\mathbf j)=
 \begin{pmatrix}a&b\\-\bar b&\bar a\end{pmatrix},
 \qquad a,b\in\C.
\]
The multiplication rules in $\HH$ show that $\phi$ is a
homomorphism of real algebras.
It is injective because its first row determines $a$ and $b$.
Applying $\phi$ to each matrix entry gives an embedding
$\Phi:H\hookrightarrow G$. For a square matrix $B$, write $[B]_m=\diag(B,\ldots,B)$ with $m$ diagonal blocks.
Moreover, we set $$J_0=\begin{pmatrix}0&1\\-1&0\end{pmatrix} \text{ and } J=[J_0]_m.$$
Defining $\theta(g)=J\bar gJ^{-1}$ for $g\in G$, we can easily see that
a matrix $B\in\operatorname{Mat}_2(\C)$ satisfies
$B=J_0\bar BJ_0^{-1}$ exactly when $B=\phi(q)$ for some
$q\in\HH$. Applying this calculation to every block gives $\Phi(H)=G^\theta$.
Therefore, we can identify $H$ with its image under $\Phi$.
We give $G$ and $H$ their natural structures as reductive Nash
groups in the sense of \cite{Sun2015}.

Representations of reductive groups are smooth admissible
Fr\'echet representations of moderate growth, as in
Chapter 11 of \cite{WallachII}.
This is the convention used in Section 2.2 of \cite{MOYglobal}.
All invariant linear forms are required to be continuous.
A representation $V$ of $G$ is distinguished by $H$ if
$\Hom_H(V,\C)\ne0$.

Let $P=MU$ be the upper triangular minimal parabolic of $G$, where $M$ is the diagonal subgroup and $U$ is the upper
triangular unipotent subgroup.
For a closed subgroup $Q$ of $G$, let $\delta_Q$ be its modulus
character.

For $z\in\C^\times$, put $\nu(z) := |z|_\C=z\bar z$.
For a character $\delta$ of $\C^\times$ and $s\in\C$, set
$\delta[s]=\delta\nu^s$. Moreover, let $r(\delta)$ be the unique real number for which
$\delta[-r(\delta)]$ is unitary, and let $\iota(z) := \bar z$.
Then we can define $\delta^*=(\delta^\iota)^\vee$, and note that
$\delta^*(z)=\delta(\bar z)^{-1}$. Also $\delta^\iota=\delta\circ\iota$ and
$\delta^\vee=\delta^{-1}$.
For $k\in\mathbb Z$, define $\omega_k(z)=(z/\sqrt{z\bar z})^k$.

Let $\delta_1,\ldots,\delta_{2m}$ be characters of $\C^\times$
satisfying $r(\delta_1)\ge\cdots\ge r(\delta_{2m})$.
For each $i$, let $k_i$ and $\lambda_i$ be the unique parameters
such that
\begin{equation}\label{eq:characters}
 \delta_i=\omega_{k_i}[\lambda_i],\qquad
 k_i\in\mathbb Z,\quad\lambda_i\in\C.
\end{equation}
The definition of $r(\delta_i)$ gives $r(\delta_i)=\Re\lambda_i$.
Set $\sigma=\delta_1\otimes\cdots\otimes\delta_{2m}$.
We realize the normalized induction $\IP(\sigma)$ on the space
of smooth functions $f:G\to\C$ satisfying
\[
 f(utg)=\delta_P(t)^{1/2}\sigma(t)f(g),
 \qquad u\in U,\ t\in M,\ g\in G.
\]
The group $G$ acts by right translation: $(g_0\cdot f)(g)=f(gg_0)$.
The corresponding standard module is
\begin{equation}\label{eq:standard}
 S=\delta_1\times\cdots\times\delta_{2m}=\IP(\sigma).
\end{equation}
Every standard module of $G$ has this form.
Indeed, $\GL_k(\C)$ has no representations that are essentially
square integrable modulo its center when $k>1$; see
Section 2.6 of \cite{MOYglobal}.
Let $S_{2m}$ be the permutation group of $\{1,\ldots,2m\}$.

\begin{theorem}\label{thm:main}
The module $S$ is distinguished by $H$ if and only if there
exists an involution $p\in S_{2m}$ without fixed points such that
$\delta_{p(i)}=\delta_i^*$ for every $i$.
Moreover,
\begin{equation}\label{eq:bound}
 \dim\Hom_H(S,\C)
 \le\#\{p\in S_{2m}:p^2=1,\ p(i)\ne i,\
                  \delta_{p(i)}=\delta_i^*\ \text{for every }i\}.
\end{equation}
\end{theorem}

Theorem~\ref{thm:main} proves the Archimedean case (Gal1) of Theorem 4.6 in \cite{MOYglobal}.
The main step is the verification of the hypotheses of Theorem 5.8 in \cite{ST}.

\section{Proof of the Main Theorem}

Define $X=\{x\in G:x=\theta(x)^{-1}\}$.
The action of $G$ on $X$ is $g\cdot x=gx\theta(g)^{-1}$.
Let $e$ be the identity element of $G$.
Since the stabilizer of $e$ is $H$, the map $gH\mapsto g\cdot e$
identifies $G/H$ with $G\cdot e$.
For $x\in X$, define $\theta_x=\Ad(x)\circ\theta$.
For a subgroup $Q$ of $G$, put $Q_x=Q\cap G^{\theta_x}$.

Let $\alpha=(1,\ldots,1)$ have $2m$ entries.
By \cite{MOYglobal}, the orbits of $P$ in $G\cdot e$
are indexed by
\[
 J(\alpha)
 =\left\{s=(a_{i,j})\in\operatorname{Mat}_{2m}(\mathbb Z_{\ge0}):
 s={}^ts,\ \sum_j a_{i,j}=1,\ a_{i,i}\in2\mathbb Z\right\}.
\]
The row sums and symmetry make each $s$ a permutation matrix.
Its symmetry implies that the corresponding permutation $p$
satisfies $p^2=1$.
For each $i$, the defining conditions give
\begin{equation}\label{eq:no-fixed-points}
 0\le a_{i,i}\le\sum_j a_{i,j}=1,
 \qquad a_{i,i}\in2\mathbb Z.
\end{equation}
Thus $a_{i,i}=0$.
Since $p(i)=i$ would give $a_{i,i}=1$, the permutation $p$ has
no fixed points.
Write $w_s=s$ for this permutation matrix viewed as an element
of $G$.
Set $$w_\star=\Bigg[\begin{pmatrix}0&1\\1&0\end{pmatrix}\Bigg]_m.$$
We have that the representatives in Section 4.1 of \cite{MOYglobal} satisfy
$x_s\in Mw_sw_\star\cap G\cdot e$.

Write $x_s=a_sw_sw_\star$ with $a_s\in M$.
For $t\in M$, we have $\theta(t)=w_\star\bar t\,w_\star^{-1}$.
Since $M$ is abelian, substitution gives
\[
 \theta_{x_s}(t)
 =a_sw_s\bar t\,w_s^{-1}a_s^{-1}
 =w_s\bar t\,w_s^{-1}.
\]
Consequently,
\begin{equation}\label{eq:torus}
 \begin{aligned}
 \theta_{x_s}(\diag(z_1,\ldots,z_{2m}))
   &=\diag(\bar z_{p(1)},\ldots,\bar z_{p(2m)}),\\
 M_{x_s}&=\{\diag(z_i):z_i\in\C^\times,\ z_{p(i)}=\bar z_i\}.
 \end{aligned}
\end{equation}
Define $\mathfrak s=\{\diag(t_i):t_i\in\R\}$.
Formula~\eqref{eq:torus} shows that $\theta_{x_s}$ preserves $M$.
Its differential preserves $\mathfrak s$.

\begin{lemma}\label{lem:ST}
For $\sigma$ in \eqref{eq:standard},
\begin{equation}\label{eq:estimate}
 \dim\Hom_H(\IP(\sigma),\C)
 \le\sum_{s\in J(\alpha)}\dim\Hom_{M_{x_s}}(\sigma,\C).
\end{equation}
\end{lemma}

\begin{proof}
We check the standing assumptions of Section 5.2 in \cite{ST}.
The symbols $L,N,\gamma$ in ST denote our $M,U,\sigma$,
respectively.
Since $G$ is connected, it is of inner type.
The subgroup $H\simeq G^\theta$ is symmetric.
The character $\sigma$ has finite length.

Define the Cartan involution $\theta_c(g)=(\bar g^{\,t})^{-1}$.
A direct computation gives $\theta_c\theta=\theta\theta_c$.
The algebra $\mathfrak s$ is a maximal abelian subspace of
$\Lie(G)^{-\theta_c}$.
The involution $\theta$ preserves $\mathfrak s$.
The positive roots determined by the upper triangular matrices
make $P$ standard.

We now check conditions (A) through (D) of
\cite[Section~5.3]{ST}.
For (A), put $\mathfrak q=\mathfrak{su}^*(2m)$.
The simple component of the symmetric Lie algebra pair is
\[
 \bigl(\mathfrak{sl}_{2m}(\C)_\R,\mathfrak{su}^*(2m)\bigr)
 \simeq(\mathfrak q_\C,\mathfrak q).
\]
Here $\mathfrak q_\C$ is regarded as a real Lie algebra.
This pair has type (A)(1).
For $m=1$, the same description holds with
$\mathfrak q=\mathfrak{su}(2)$.

For (B), let $M_c$ be the diagonal subgroup with entries in $S^1$.
Let $A$ be the diagonal subgroup with entries in $\R_{>0}$.
Then $M=M_cA$ and $P=M_cAU$.
Since $M_c$ is compact, $P$ is cuspidal.
Using the parameters in \eqref{eq:characters}, set
$\sigma_0=\bigotimes_i\omega_{k_i}$.

Use the norm coordinates $\mathfrak a_M=\R^{2m}$ from \cite{MOYglobal}.
Let $\mathfrak a_{M,\C}^*$ be its complexified dual.
In these coordinates, the map $H_M:M\to\mathfrak a_M$ is
\[
 H_M(\diag(z_i))=(\log|z_i|_\C)_i.
\]
Let $\lambda=(\lambda_i)\in\mathfrak a_{M,\C}^*$.
Define $\sigma_0[\lambda](t)=e^{\langle\lambda,H_M(t)\rangle}\sigma_0(t)$
for $t\in M$.
The character identities above give $\sigma=\sigma_0[\lambda]$.

Put $\tau=\sigma_0|_{M_c}$ and $\Lambda=\lambda\circ dH_M$.
Since $\log|e^{t_i}|_\C=2t_i$, we have
\begin{equation}\label{eq:parameter}
 \Lambda(\diag(t_i))=2\sum_i\lambda_i t_i.
\end{equation}
Thus $\sigma=\tau\boxtimes e^\Lambda$ on $M_cA$.
The character $\tau$ is square integrable because $M_c$ is compact.
For $i<j$, let $\beta_{ij}(\diag(t_i))=t_i-t_j$.
Let $E_{ii}$ be the $i$th diagonal matrix unit.
The dual of $\beta_{ij}$ with respect to the Killing form on the
semisimple part is $X_{\beta_{ij}}=c_m(E_{ii}-E_{jj})$ for a
constant $c_m>0$.
The ordering of the $r(\delta_i)$ gives
\begin{equation}\label{eq:chamber}
 \Re\Lambda(X_{\beta_{ij}})
 =2c_m\bigl(r(\delta_i)-r(\delta_j)\bigr)\ge0.
\end{equation}
This verifies the chamber condition in (B).
Condition (C) is vacuous because $\Lie(M)$ is abelian.
Condition (D) holds because the target character is trivial.

Choose $\eta_s\in G$ such that $x_s=\eta_s\cdot e$.
Inversion identifies $P\backslash G/H$ with $H\backslash G/P$.
Thus $g_s=\eta_s^{-1}$ gives the representatives required in ST.
Their associated involutions satisfy
\[
 \Ad(g_s^{-1})\circ\theta\circ\Ad(g_s)
 =\Ad\bigl(\eta_s\theta(\eta_s)^{-1}\bigr)\circ\theta
 =\theta_{x_s}.
\]
Formula~\eqref{eq:torus} verifies the required preservation of
$M$ and $\mathfrak s$.
The fixed subgroup of $M$ is $M_{x_s}$.
The vanishing statement in \cite[Theorem~5.8]{ST} excludes
contributions from derivatives of positive order in directions
normal to the orbits.
The multiplicity estimate in that theorem therefore gives
\eqref{eq:estimate}.
\end{proof}

\begin{proof}[Proof of Theorem~\ref{thm:main}]
Assume that $\Hom_H(S,\C)\ne0$.
By Lemma~\ref{lem:ST}, there exists $s\in J(\alpha)$ such that
$\Hom_{M_{x_s}}(\sigma,\C)\ne0$.
Let $p\in S_{2m}$ be the permutation corresponding to $s$. By \eqref{eq:no-fixed-points}, $p$ has no fixed points. Since $\sigma$ is a character, the nonzero invariant form implies
$\sigma|_{M_{x_s}}=1$.
By \eqref{eq:torus}, the coordinates $z_i$ with $i<p(i)$ vary
independently on $M_{x_s}$.
Therefore
\[
 \sigma|_{M_{x_s}}=1
 \quad\Longleftrightarrow\quad
 \delta_i(z)\delta_{p(i)}(\bar z)=1
 \quad(i<p(i),\ z\in\C^\times).
\]
By the definition of $\delta_i^*$, this condition is equivalent
to $\delta_{p(i)}=\delta_i^*$ for every $i$.
This proves necessity.

The same character calculation applies to every $s\in J(\alpha)$.
The corresponding summand in \eqref{eq:estimate} has dimension $1$ if the pairing condition holds, and $0$ otherwise.
Summing over $J(\alpha)$ gives \eqref{eq:bound} by Lemma~\ref{lem:ST}.

Conversely, let $p\in S_{2m}$ be an involution without fixed points
such that $\delta_{p(i)}=\delta_i^*$ for every $i$.
Let $s\in J(\alpha)$ be its permutation matrix.
Use the character parameters defined in \eqref{eq:characters}.
Set $\sigma_0=\bigotimes_i\omega_{k_i}$ and $\lambda=(\lambda_i)$.
For $k\in\mathbb Z$ and $\mu\in\C$, the identities
$\omega_k(\bar z)=\omega_k(z)^{-1}$ and $\nu(\bar z)=\nu(z)$
give $\omega_k[\mu]^*=\omega_k[-\mu]$.
Uniqueness of the character parameters therefore gives
$k_{p(i)}=k_i$ and $\lambda_{p(i)}=-\lambda_i$.
The equality $k_{p(i)}=k_i$ implies $\sigma_0|_{M_{x_s}}=1$.
By \eqref{eq:torus}, the action of $\theta_{x_s}$ on
$\mathfrak a_{M,\C}^*$ permutes the coordinates by $p$.
The subspace on which this action is multiplication by $-1$ is
\[
 (\mathfrak a_{M,\C}^*)_{x_s}^-
 =\{\mu=(\mu_i):\mu_{p(i)}=-\mu_i\}.
\]
Hence $\lambda\in(\mathfrak a_{M,\C}^*)_{x_s}^-$.

By \cite[Lemma~4.2]{MOYglobal},
\[
 \left.\delta_{P_{x_s}}\delta_P^{-1/2}\right|_{M_{x_s}}=1.
\]
Since $\sigma_0$ is unitary, it satisfies
\cite[Assumption~2.2]{MOYlocal}.
Theorem~5.4 of \cite{MOYlocal} therefore gives a nonzero
continuous form $\ell$ on $\IP(\sigma_0[\lambda])=S$ that is
invariant under $G_{x_s}$.

Choose $\eta_s\in G$ such that $x_s=\eta_s\cdot e$.
Let $\rho$ denote the action of $G$ on $S$.
Define $\ell_H(v)=\ell(\rho(\eta_s)v)$.
Since $G_{x_s}=\eta_sH\eta_s^{-1}$, for $h\in H$ we have
\[
 \ell_H(\rho(h)v)
 =\ell\bigl(\rho(\eta_sh\eta_s^{-1})\rho(\eta_s)v\bigr)
 =\ell_H(v).
\]
The form $\ell_H$ is continuous as a composition of continuous maps, and it is nonzero as $\rho(\eta_s)$ is invertible. Therefore, we get that $S$ is distinguished by $H$, as desired.
\end{proof}

\section{Acknowledgements}
We would like to thank our advisor, Omer Offen, for proposing this problem, and for helpful discussions.

\end{document}